\documentclass[a4paper,11pt]{article}
\usepackage{latexsym}
\usepackage{amssymb}
\usepackage{enumerate}
\usepackage{enumitem}
\usepackage{xcolor}
\usepackage{commath}
\usepackage{amsfonts}
\usepackage{amsmath,amsthm}
\usepackage{hyperref}
\usepackage{algorithm}
\usepackage{algorithmic}
\usepackage{framed}
\usepackage{boites}
\usepackage{graphicx}
\usepackage{cases}

\newenvironment{@abssec}[1]{%
    \if@twocolumn

      \section*{#1}%
    \else

      \vspace{.05in}\footnotesize
      \parindent .2in
 {\upshape\bfseries #1. }\ignorespaces
    \fi}

    {\if@twocolumn\else\par\vspace{.1in}\fi}

\newenvironment{keywords}{\begin{@abssec}{\keywordsname}}{\end{@abssec}}

\newenvironment{AMS}{\begin{@abssec}{\AMSname}}{\end{@abssec}}

\newcommand\keywordsname{Key words}
\newcommand\AMSname{AMS subject classifications}
\newcommand\AMname{AMS subject classification}
\newcommand\restr[2]{{% we make the whole thing an ordinary symbol
\left.\kern-\nulldelimiterspace % automatically resize the bar with \right
#1 % the function
\vphantom{|} % pretend it's a little taller at normal size
\right|_{#2} % this is the delimiter
}}
\newtheorem{theorem}{Theorem}[section]
\newtheorem{lemma}[theorem]{Lemma}

\newtheorem{proposition}[theorem]{Proposition}
\newtheorem{remark}[theorem]{Remark}
\newtheorem{definition}[theorem]{Definition}

\newtheorem{test}{Theorem}

\newtheorem{thm}{Theorem}

\newcommand{\NN}{\mathbb{N}}

\newcommand{\RR}{\mathbb{R}}

\renewcommand{\SS}{\mathbb{S}}

\def\XXint#1#2#3{{\setbox0=\hbox{$#1{#2#3}{\int}$}
\vcenter{\hbox{$#2#3$}}\kern-.5\wd0}}

\newcommand{\link}{\mathop{\circ\kern-.35em -}}

\newcommand{\ol}{\overline}
\newcommand{\pa}{\partial}

\newcommand{\dv}{\mathop{\mathrm{div}}}

\newcommand{\gr}{\nabla}

\newcommand{\al}{\alpha}
\newcommand{\be}{\beta}
\newcommand{\ga}{\gamma}

\newcommand{\De}{\Delta}
\newcommand{\ve}{\varepsilon}
 
\newcommand{\la}{\lambda}
\newcommand{\La}{\Lambda}

\newcommand{\Si}{\Sigma}

\newcommand{\om}{\omega}
\newcommand{\Om}{\Omega}
\newcommand{\rn}{{\mathbb{R}}^N}
\newcommand{\sg}{\sigma}

\newcommand{\cC}{\mathcal{C}}

\newcommand{\cF}{\mathcal{F}}
\newcommand{\cG}{{\mathcal G}}
\newcommand{\cH}{{\mathcal H}}

\newcommand{\cU}{{\mathcal U}}
\newcommand{\cX}{\mathcal{X}}
\newcommand{\cY}{\mathcal{Y}}

\title{Local analysis of a two phase free boundary problem concerning mean curvature
\thanks{This research was partially supported by the Grant-in-Aid for JSPS Fellows No.18J11430.}}

\author{Lorenzo Cavallina %\thanks{
}
\date{}

\begin{document}

\maketitle

\begin{abstract}
We consider an overdetermined problem for a two phase elliptic operator in divergence form with piecewise constant coefficients. We look for domains such that the solution $u$ of a Dirichlet boundary value problem also satisfies the additional property that its normal derivative $\pa_n u$ is a multiple of the radius of curvature at each point on the boundary.
When the coefficients satisfy some ``non-criticality" condition, we construct nontrivial solutions to this overdetermined problem employing a perturbation argument relying on shape derivatives and the implicit function theorem. Moreover, in the critical case, we employ the use of the Crandall-Rabinowitz theorem to show the existence of a branch of symmetry breaking solutions bifurcating from trivial ones.
Finally, some remarks on the one phase case and a similar overdetermined problem of Serrin type are given.
\end{abstract}



\pagestyle{plain}
\thispagestyle{plain}

\section{Introduction}\label{introduction}
\subsection{Problem setting and known results}
Let $\Om\subset\rn$ ($N\ge2$) be a bounded domain of class $\cC^{2}$ and $D\subset\ol D \subset \Om$ be an open set of class $\cC^2$ with at most finitely many connected components such that $\Om\setminus D$ is connected.  
Moreover, let $n$ denote the outward unit normal vector to both  $\pa\Om$ and $\pa D$ and let their mean curvature $H$ be defined as the tangential divergence of the normal vector $n$, that is $\dv_\tau(n)$ (notice that, under this definition, the mean curvature of a ball of radius $R$ is equal to $(N-1)/R$ everywhere on the boundary). 
%\textcolor{red}{define $\sg_c$ and $\sg$}
Given a positive constant $\sg_c\ne 1$, define the following piecewise constant function: 
\begin{equation}\label{sigma}
    \sg(x)=\sg_c\ \cX_D+\cX_{\Om\setminus D},
\end{equation}
where $\cX_A$ is the characteristic function of the set $A$ (i.e., $\cX_A(x)$ is $1$ if $x\in A$ and $0$ otherwise).
The aim of this paper is to study the following two phase overdetermined problem. Find those pairs of domains $(D,\Om)$ with the properties stated in the introduction such that the mean curvature $H$ of $\pa\Om$ never vanishes and such that the following overdetermined problem admits a solution for some real constant $d>0$.  
%\begin{equation}\label{odp}
    \begin{numcases}{}
    -\dv\left(\sg\gr u\right)=1 \quad \text{in }\Om, \label{eq}\\
    u=0\quad \text{on }\pa\Om,\label{bc}\\
    \pa_n u=-d/H \quad \text{on }\pa\Om\label{oc},
    \end{numcases}
%\end{equation}
where %$d$ is a given positive constant and 
$\pa_n$ stands for the outward normal derivative at the boundary.
\begin{figure}[h]
\centering
\includegraphics[width=0.55\linewidth]{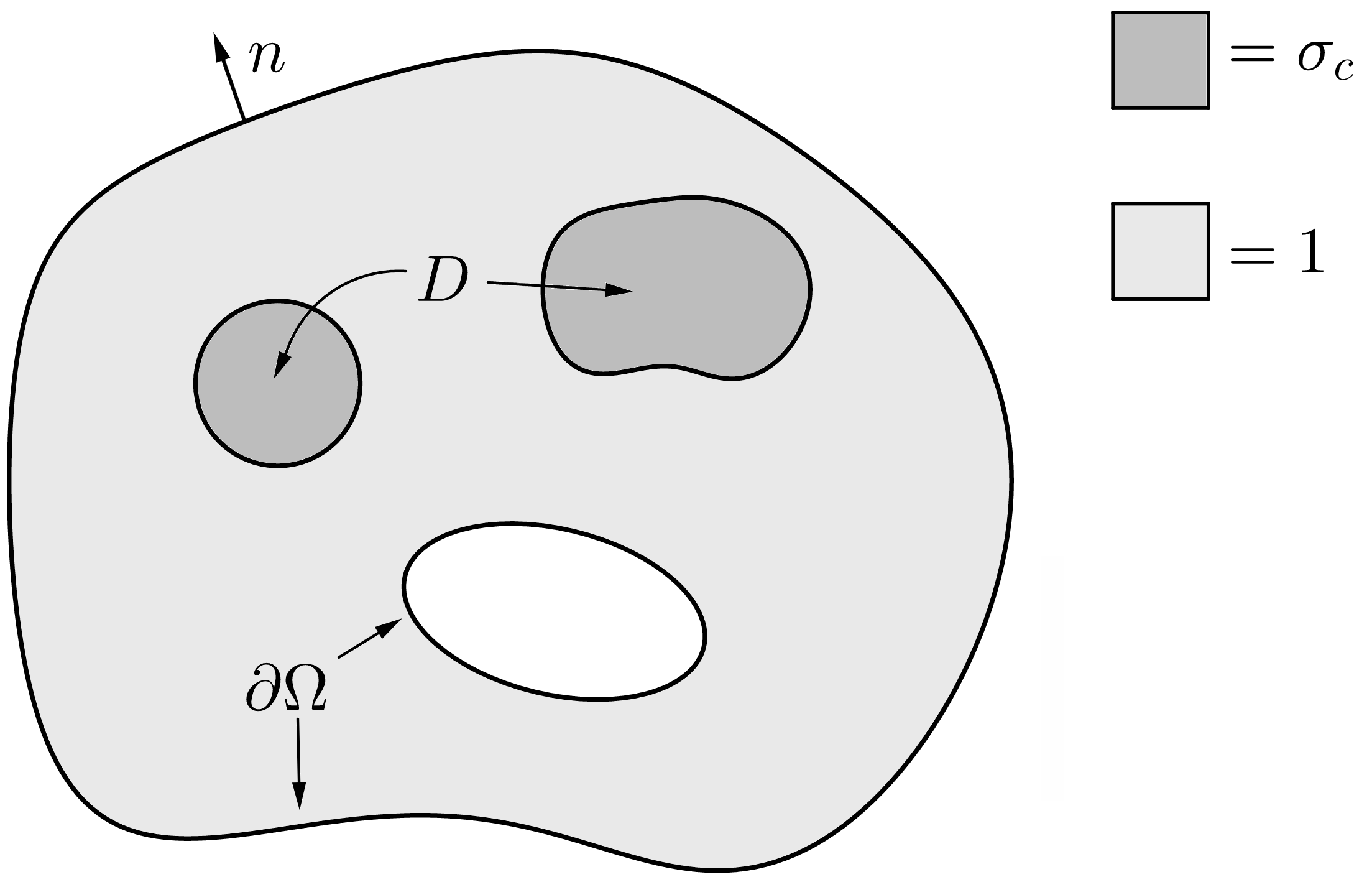}
\caption{Problem setting} 
\label{pb setting}
\end{figure}

One of the most famous (and influential) results concerning one phase overdetermined problems is due to Serrin \cite{Se1971}. In particular, he showed that balls are the only domains such that the value $u(x)$ and the normal derivative $\pa_nu(x)$ of the solution $u$ of some given elliptic problem (for the Laplace operator) both attain a constant value on the boundary (see Theorem \ref{thm serrin}, later in this paper).
Many mathematicians, inspired by Serrin's theorem, have extended his results and given alternative proofs: we refer the interested reader to the survey papers \cite{Magna aswr, NT2018} and the references therein. In particular, similar overdetermined problems corresponding to nonconstant overdetermined conditions (such as \cite{BHS2014}) and the corresponding stability properties have been considered (see \cite{BNST2008, MagnaniniPoggesipp2017, MagnaniniPoggesi2019, MagnaniniPoggesi2020}). 

On the other hand, two phase overdetermined problems (that is, overdetermined problems related to an operator in divergence form like the one in \eqref{eq}) show a more diverse behavior. Indeed, depending on the setting, the solutions of such overdetermined problems might enjoy a high degree of symmetry just as in Serrin's original work (see \cite{KLeeS, Sak2016?, Sak bessatsu, camasa, Sak2019?, CSU}) or allow for the existence of nontrivial (nonsymmetric) solutions, due to the interaction between the geometry of the two phases $D$ and $\Om\setminus D$ (see \cite{KLS, camasa, CY1, CYisaac}).

\subsection{Main results for overdetermined problem \eqref{eq}--\eqref{oc}}
Overdetermined problem \eqref{eq}--\eqref{oc} has the following interpretation. If one considers $\sg_c$ to be a dimensionless quantity, then, a quick dimensional analysis yields that the solution $u$ of the boundary value problem \eqref{eq}--\eqref{bc} has the dimension of length squared. As a consequence, its normal derivative $\pa_n u$ has the dimension of length. Overdetermined condition \eqref{oc} then translates to requiring that the value $\pa_n u(x)$ be directly proportional to the radius of curvature at each point $x\in\pa\Om$.

First of all, it is important to notice that, unlike the boundary value problem \eqref{eq}--\eqref{bc}, the overdetermined problem \eqref{eq}--\eqref{oc} is not solvable for all pairs $(D,\Om)$. In what follows, when no confusion arises, we will also refer to the very pairs of domains $(D,\Om)$ that make problem \eqref{eq}--\eqref{oc} solvable as \emph{solutions} of \eqref{eq}--\eqref{oc}. In particular, notice that, for all values of $\sg_c>0$, any pair of concentric balls $(D_0,\Om_0)$ is a solution of \eqref{eq}--\eqref{oc} corresponding to $d=\frac{N-1}{N}$. We will refer to such pairs as \emph{trivial solutions}. By a scaling argument, it is enough to check this fact when $\Om_0$ is the unit ball centered at the origin and $D_0$ is the concentric ball with radius $0<R<1$. Under these assumptions, the unique solution to \eqref{eq}--\eqref{bc} is given by
\begin{equation}\label{u}
u(x)=\begin{cases}\displaystyle
\frac{1-R^2}{2N%\sg_s
}+\frac{R^2-|x|^2}{2N\sg_c}&\quad |x|\in[0,R),\\
\vspace*{-6mm}\\
\displaystyle\frac{1-|x|^2}{2N%\sg_s
}&\quad|x|\in [R,1],
\end{cases}
\end{equation}
and also satisfies the overdetermined condition \eqref{oc} for $d=\frac{N-1}{N}$.
%\begin{remark}\label{really trivial}
%It goes without saying that when $\sg_c=1$, the operator $\dv(\sg\gr\cdot)=\De$ is not affected by the geometrical shape of the core $D$. Therefore, we decide to use the term ``trivial solution" to include also the case of pairs of the form $(D,\Om)$ where $\Om$ is a ball and $\sg_c=1$. 
%\end{remark}

Notice that the limit case $R=0$ in the above corresponds to the pair $(\emptyset,\Om_0)$, which, for the purpose of this paper, will still be considered a \emph{trivial} solution of \eqref{eq}--\eqref{oc}.
\begin{test}\label{thm I}
Problem \eqref{eq}--\eqref{oc} has no solution if $d>\frac{N-1}{N}$. Moreover, if $d=\frac{N-1}{N}$ then the only solutions of \eqref{eq}--\eqref{oc} are trivial.
\end{test}
The one phase analogue of overdetermined problem \eqref{eq}--\eqref{oc} in the critical case $d=\frac{N-1}{N}$ was studied by Magnanini and Poggesi in \cite{MagnaniniPoggesi2019}. The authors also showed stability estimates by means of integral inequalities. 

In what follows, we will let the quantity $d\le\frac{N-1}{N}$ vary and study the nontrivial solutions of problem \eqref{eq}--\eqref{oc} that are obtained by a small perturbation of trivial ones. 
Let $(D_0,\Om_0)$ denote a pair of concentric balls centered at the origin. In what follows, pairs of perturbed domains, denoted by $(D_f,\Om_g)$, will be parametrized by functions $f\in\cF$, $g\in\cG$, where
\begin{equation}\label{FG}
    \cF=\left\{f\in\cC^{2,\al}(\pa D_0)\;:\; \int_{\pa D_0} f =0\right\} \quad \text{and}\quad 
    \cG=\left\{g\in\cC^{2,\al}(\pa \Om_0)\;:\; \int_{\pa \Om_0} g =0\right\}.
\end{equation}
%$f\in\cC^{2,\al}(\pa D_0)$, $ g\in\cC^{2,\al}(\pa\Om_0)$. 
If the functions $f$ and $g$ are sufficiently small, the perturbed domains $D_f$ and $\Om_g$ are well defined as the unique bounded domains whose boundaries are 
\begin{equation}\label{perturbed domains}
    \pa D_f=\Big\{x+f(x)n(x) \;:\; x\in\pa D_0\Big\} \quad \text{and}\quad
    \pa \Om_g=\Big\{x+g(x)n(x)\;:\; x\in\pa\Om_0\Big\}.
\end{equation}
%\textcolor{red}{define $\Phi$}
In order to study the nontrivial solutions of \eqref{eq}--\eqref{oc}, we will employ the use of a mapping 
\begin{equation*}
\cF\times\cG\times(0,\infty)\ni(f,g,s)\mapsto    \Phi(f,g,s)
\end{equation*}
that vanishes if and only if the pair $(D_f,\Om_g)$ is a solution to problem \eqref{eq}--\eqref{oc} when $\sg_c=s$. The precise definition of $\Phi$ will be given in \eqref{Phi}. 
We will show that nontrivial solutions of \eqref{eq}--\eqref{oc} near the trivial solution show different behaviours depending on the value of $\sg_c$. Here we define the \emph{critical values}: 
\begin{equation}\label{s_k}
   s_k = \frac{(k+N-2)\left(k+N-1+(k-1)R^{2-N-2k}\right)}{(k+N-2)(k+N-1)-k(k-1)R^{2-N-2k}}
\end{equation}
Notice that, for all integers $k$ that verify
\begin{equation}\label{nondegeneracy}
%k\ge2 \quad \text{and}\quad
{k(k-1)}R^{2-N-2k}<{(k+N-2)(k+N-1)},
\end{equation}
the expression \eqref{s_k} yields a well defined positive real number $s_k$. 
We will use the notation $\Si$ to denote the following set of critical values:
\begin{equation}\label{Si}
    \Si=\Big\{ s_k \;:\; k\in\NN \text{ verifies $k\ge2$ and  \eqref{nondegeneracy}} \Big\}\subset(0,\infty).
\end{equation}
The following theorem is obtained by applying the implicit function theorem for Banach spaces (see Theorem \ref{ift}, page \pageref{ift}) to the function $\Phi$ when $\sg_c$ is not a critical value. 
\begin{test}\label{thm II}
Let $\sg_c\notin \Si$. Then, there exists a threshold $\ve>0$ such that, for all $f\in\cF$ satisfying $\norm{f}_{\cC^{2,\al}}<\ve$ there exists a function $g=g(f,\sg_c)\in\cG$ such that the pair $(D_f,\Om_g)$ is a solution to problem \eqref{eq}--\eqref{oc} for some $d\le\frac{N-1}{N}$. Moreover, this solution is unique in a small enough neighborhood of $(0,0)\in\cF\times\cG$. In particular, there exist infinitely many nontrivial solutions of problem \eqref{eq}--\eqref{oc}.
\end{test}

\begin{figure}[h]
\centering
\includegraphics[width=0.7\linewidth]{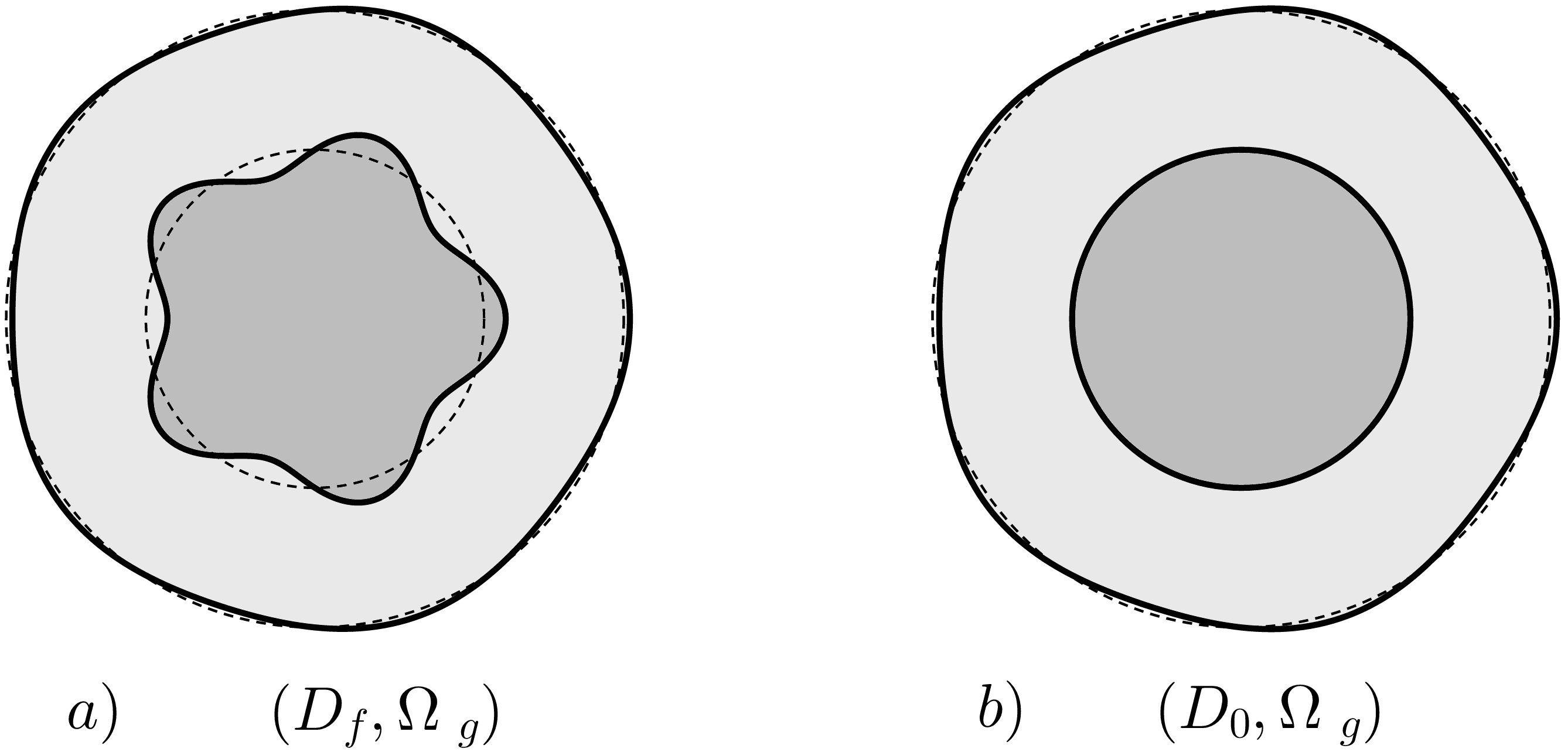}
\caption{a) A nontrivial solution for $\sg_c\notin\Si$, given by Theorem \ref{thm II}. b) A symmetry breaking solution branching from the trivial solution $(D_0,\Om_0)$ at the bifurcation point $\sg_c=s_k$, given by Theorem \ref{thm III}.}
\label{pb setting}
\end{figure}

\begin{test}\label{thm III}
Take an element $s_k\in\Si$ and consider the equation 
\begin{equation*}
\Psi(g,s)=\Phi(0,g,s) = 0,  
\end{equation*}
then $(0,s_k)\in\cG\times\RR$ is a bifurcation point of the equation $\Psi(g,s)= 0$. That is, there exists a function $\varepsilon\mapsto \la(\varepsilon)\in\RR$ with $\la(0)=0$ 
such that overdetermined problem \eqref{eq}--\eqref{oc} admits a nontrivial solution of the form $(D_0,\Om_{g(\varepsilon)})$ for $\sg_c=s_k+\la(\varepsilon)$ and $\varepsilon$ small. 
%If $N=2$, then the symmetry breaking solution $(B_R,\Om_{g(\varepsilon)})$ satisfies 
%\begin{equation}\label{symmetry breaking solutions}
 %   g(\varepsilon) = \varepsilon \cos(m\theta) + o(\varepsilon) \quad \text{as }\varepsilon\to0.   
%\end{equation}
Moreover, there exists a spherical harmonic $Y_k$ of $k$-th degree, such that the symmetry breaking solution $(D_0,\Om_{g(\varepsilon)})$ satisfies 
\begin{equation}\label{symmetry breaking solutions 2}
    g(\varepsilon) = \varepsilon Y_k(\theta) + o(\varepsilon) \quad \text{as }\varepsilon\to0.   
\end{equation}
In particular, there exist uncountably infinitely many nontrivial solutions of problem \eqref{eq}--\eqref{oc} where $D$ is a ball (spontaneous symmetry breaking solutions).
\end{test}
As the following theorem shows, the one phase case $D=\emptyset$ has a radically different behavior around trivial solutions. 
\begin{test}\label{thm IV}
Let $D=\emptyset$. Then, the trivial solutions of \eqref{eq}--\eqref{oc} are isolate solutions (in the sense of Definition \ref{isolate sol}).
\end{test}

This paper is organized as follows. In section~2, we study what happens for $d\ge\frac{N-1}{N}$ and give a proof of Theorem~\ref{thm I} by means of the Heintze--Karcher inequality and a symmetry theorem by Sakaguchi concerning a two phase overdetermined of Serrin type problem in the ball. In section~3, we prove Theorem~\ref{thm II} and construct nontrivial solutions to \eqref{eq}--\eqref{oc} by using the implicit function theorem for Banach spaces and shape derivatives. Section~4 is devoted to the study of spontaneous symmetry breaking solutions that arise when $\sg_c$ is a critical value. Here we prove Theorem~\ref{thm III} by means of the Crandall--Rabinowitz theorem. In Section~5 we show that, in the one phase setting, balls are isolated solutions (Theorem~\ref{thm IV}). Section~6 is devoted to the comparison between the overdetermined problem presented in this paper and a similar two phase overdetermined problem of Serrin type. Finally, in the appendix, we prove a technical result concerning invariant subgroups of spherical harmonics that is crucial to the proof of Theorem~\ref{thm III} in general dimension.      
\vspace*{0.5cm}
\begin{figure}[h]
\centering
\includegraphics[width=0.9\linewidth]{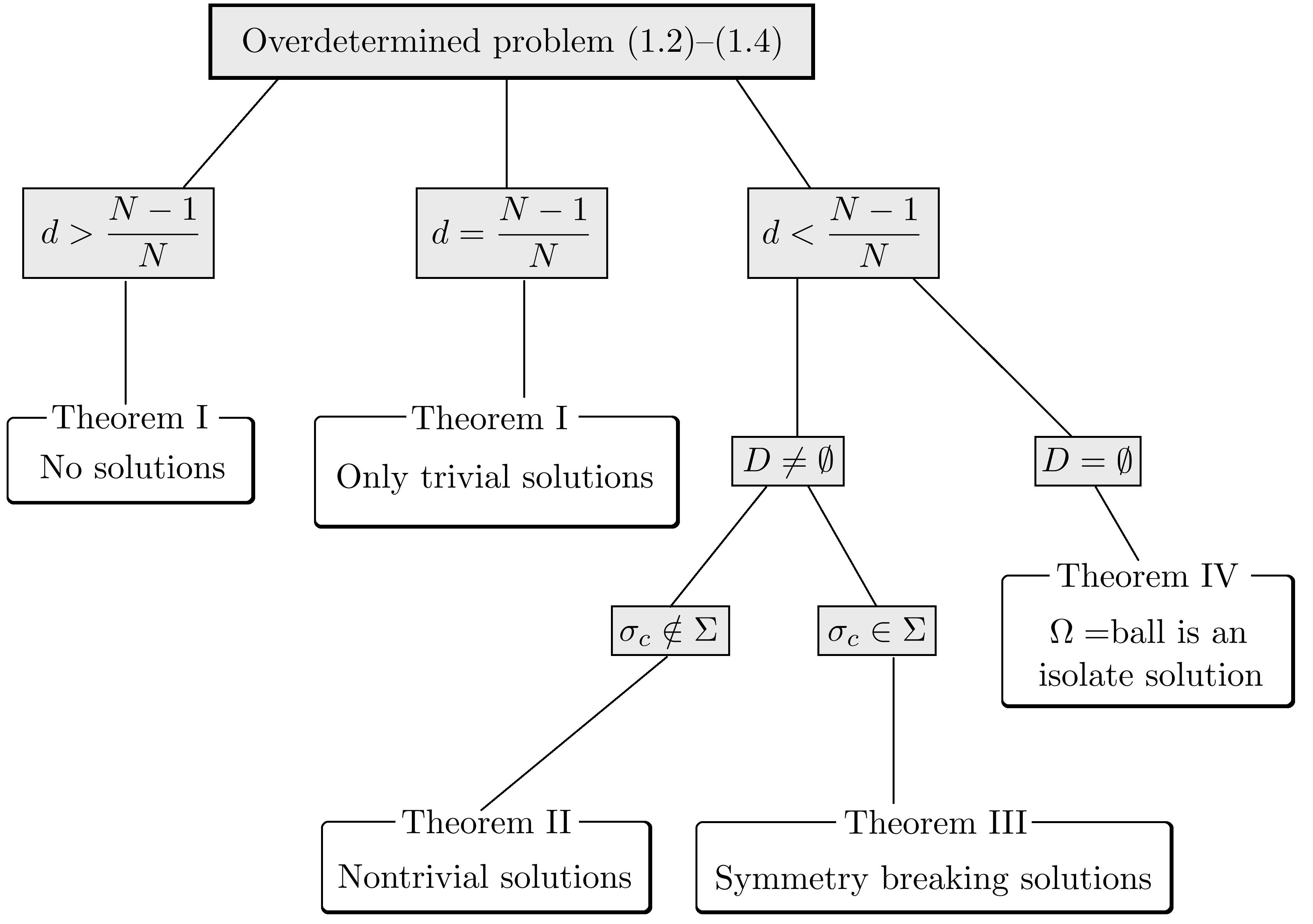}
\caption{Organization of this paper} 
\label{zu}
\end{figure}

\section{Only trivial solutions for $d=\frac{N-1}{N}$}

The proof of Theorem \ref{thm I} relies on the following two facts. 
The first is the so-called Heintze--Karcher inequality. This result was first proved for compact and embedded manifolds by Montiel and Ros in \cite{MR91} following the ideas of \cite{HK78}. It was then extended to general sets of finite perimeter in \cite{Santilli}. 
\begin{thm}[Heintze--Karcher inequality]\label{hk ineq}
Let $\Om$ be a bounded domain of class $\cC^{2}$. Then the following inequality holds
\begin{equation*}
    \int_{\pa\Om} 1/H \ge \frac{N}{N-1} |\Om|.
\end{equation*}
Moreover, equality is attained in the above if and only if $\Om$ is a ball. 
\end{thm}
The second tool that we will need concerns the following two-phase overdetermined problem of Serrin-type:
\begin{numcases}{}
    -\dv\left(\sg\gr u\right)=1 \quad \text{in }\Om, \label{eq s}\\
    u=0\quad \text{on }\pa\Om,\label{bc s}\\
    \pa_n u=-c \quad \text{on }\pa\Om\label{oc s},
\end{numcases}
where $c$ is some given positive constant.
The following theorem is a special case of Theorem 5.1 of \cite{Sak bessatsu}.
\begin{thm}[Symmetry for two phase Serrin problem in the ball, \cite{Sak bessatsu}]\label{thm saka}
Let $\Om\subset\rn$ be an open ball and let $\emptyset\ne D\subset\ol D \subset \Om$ be an open set of class $\cC^2$, with at most finitely many connected components, and such that $\Om\setminus D$ is connected. Now, if $(D,\Om)$ is a solution of the overdetermined problem \eqref{eq s}--\eqref{oc s}, then $D$ must be a ball concentric with $\Om$.  
\end{thm}

\begin{proof}[Proof of Theorem \ref{thm I}]
Let $(D,\Om)$ be a solution of problem \eqref{eq}--\eqref{oc} for some positive real number $d$. 
We recall that, since both $D$ and $\Om$ are of class $\cC^2$, %then equation \eqref{eq}
the boundary value problem \eqref{eq}--\eqref{bc} admits the following equivalent formulation as a transmission problem (\cite[Theorem 1.1]{athastra96}): 
\begin{equation}\label{tp}
\begin{cases}
-\sg_c\De u =1 \quad \text{in } D,\\
-\De u = 1 \quad \text{in } \Om\setminus\ol D,\\
[u]=[\sg\pa_n u]=0 \quad \text{on } \pa D,\\
u=0 \quad \text{on }\pa\Om,
\end{cases}
\end{equation}
where brackets are used to denote the jump of a quantity along the interface $\pa D$.

By \eqref{tp}, with the aid of the divergence theorem, we get
\begin{equation*}
    |\Om|=\int_D 1+ \int_{\Om\setminus\ol D} 1
 = -\sg_c \int_D \De u - \int_{\Om\setminus\ol D} \De u = -\int_{\pa D} [\sg \pa_n u] -\int_{\pa\Om} \pa_n u = -\int_{\pa\Om} \pa_n u. 
 \end{equation*}
Now, the overdetermined condition \eqref{oc} and Heintze--Karcher inequality yield 
\begin{equation*}
    |\Om|= -\int_{\pa \Om}\pa_n u = d\int_{\pa\Om} 1/H\ge \frac{d N}{N-1} |\Om|.  
\end{equation*}
In particular, since $|\Om|>0$, this implies that $d\le \frac{N-1}{N}$, as claimed. 
Moreover, when $d=\frac{N-1}{N}$, then we deduce that $\Om$ is a ball by the second part of Theorem \ref{hk ineq}. In particular, $H$ is constant on $\pa\Om$ and hence, the solution $u$ of problem \eqref{eq}--\eqref{oc} also solves the overdetermined condition \eqref{oc s} for some constant $c>0$. We conclude by Theorem~\ref{thm saka}.
\end{proof}

\begin{remark}
Notice that, since $\pa_n u < 0$ on $\pa\Om$ by Hopf's lemma, if the pair $(D,\Om)$ is a solution of \eqref{eq}--\eqref{oc}, then $\Om$ must be a strictly mean convex set (that is, $H>0$ on $\pa\Om$).  
\end{remark}

\section{Local existence of nontrivial solutions for $\sg_c\notin\Si$}
The proof of Theorem \ref{thm II} will rely on the following version of the implicit function theorem for Banach spaces (see \cite[Theorem 2.3, page 38]{AP1983} for a proof). 

%\begin{thm}[Implicit function theorem]\label{ift}
%Suppose that $\mathcal{X}$, $\mathcal{Y}$ and $\mathcal{Z}$ are three Banach spaces, $\mathcal{U}$ is an open subset of $\mathcal{X}\times\mathcal{Y}$, $(x_0,y_0)\in \mathcal{U}$, and $\Psi:\mathcal{U}\to \mathcal{Z}$ is a Fr\'echet differentiable mapping such that $\Psi(x_0,y_0)=0$. Assume that the partial derivative $\pa_y \Psi(x_0,y_0)$ with respect to the variable $y$ at $(x_0,y_0)$ is a bounded invertible linear transformation from $\mathcal{Y}$ to $\mathcal{Z}$. 
%Then there exists a neighborhood $\mathcal{U}_0$ of $x_0$ in $\mathcal{X}$ and a unique continuous function $g:\mathcal{U}_0\to \mathcal{Y}$ such that $g(x_0)=y_0$, $(x, g(x))\in \mathcal{U}$ and $\Psi(x, g(x))=0$ for all $x\in\mathcal{U}$.
%Moreover, the function $g$ is Fr\'echet differentiable in $\mathcal{U}_0$ and its Fr\'echet differential $g'$ can be written as 
%\begin{equation}\label{second part}
% g^\prime (x) = - \pa_y \Psi(x, g(x))^{-1}\, \pa_x \Psi(x, g(x)) \quad \textrm{ for } x\in \mathcal{U}_0.
%\end{equation}
%\end{thm}
\begin{thm}[Implicit function theorem]\label{ift}
Let $\Psi\in\cC^k(\La\times W,Y)$, $k\ge1$, where $Y$ is a Banach space and $\La$ (resp. $U$) is an open set of a Banach space $T$ (resp. $X$). Suppose that  $\Psi(\la^*,w^*)=0$ and that the partial derivative $\pa_w\Psi(\la^*,w^*)$ is a bounded invertible linear transformation from $X$ to $Y$. 

Then there exist neighborhoods $\Theta$ of $\la^*$ in $T$ and $W^*$ of $w^*$ in $X$, and a map $g\in\cC^k(\Theta,X)$ such that the following hold:
\begin{enumerate}[label=(\roman*)]
\item $\Psi(\la,g(\la))=0$ for all $\la\in\Theta$,
\item If $\Psi(\la,u)=0$ for some $(\la,u)\in\Theta\times U^*$, then $u=g(\la)$,
\item $g'(\la)=-\left(\pa_u \Psi(p) \right)^{-1}\circ \pa_\la \Psi(p)$, where $p=(\la,g(\la))$ and $\la\in\Theta$.
\end{enumerate}
\end{thm}
\subsection{Preliminaries}
Let $\cF$ and $\cG$ be the sets defined in \eqref{FG} and let $\cH$ be the following 
\begin{equation*}
    \cH=\left\{h\in\cC^{\al}(\pa \Om_0)\;:\; \int_{\pa \Om_0} h =0\right\}.
    \end{equation*}
The sets $\cF$, $\cG$ and $\cH$ are Banach spaces,  endowed with the natural norm of the corresponding H\"older class.
We will apply the implicit function theorem to the following mapping.

\begin{equation}\label{Phi}
    \begin{aligned}
    \Phi: & \cF\times \cG\times (0,\infty) \longrightarrow \cH\\
    &(f,g,s)\longmapsto \Pi_0 \left( \pa_{n_g n_g}^2 u_{f,g}\right),
\end{aligned}
\end{equation}
where $\Pi_0:\cC^\al(\pa\Om_0)\to\cH$ is the projection operator, defined by
\begin{equation*}
\Pi_0(\varphi)=\varphi-\frac{1}{|\pa\Om_0|}\int_{\pa\Om_0}\varphi,
\end{equation*}
and $u_{f,g}$ is the solution of \eqref{eq}--\eqref{bc} with $\Om=\Om_g$ and $\sg=s\cX_{D_f}+\cX_{\Om_g\setminus D_f}$. Moreover, by a slight abuse of notation, $\pa_{n_g n_g}^2u_{f,g}$ denotes the function of value
\begin{equation}\label{function of value}
n_g\big(x+g(x)n(x)\big)\cdot \left(D^2 u_{f,g}\big(x+g(x)n(x)\big)\  n_g\big(x+g(x)n(x)\big)\right) \; \text{at any } x\in\pa\Om_0.    
\end{equation}
\begin{remark}[Zeros of $\Phi(\cdot,\cdot,\sg_c)$ correspond to solutions of \eqref{eq}--\eqref{oc}]
Notice that, by definition, $\Phi(f,g,\sg_c)=0$ if and only if the function defined in \eqref{function of value} takes the same value for all $x\in\pa\Om_0$, that is, if and only if $\pa_{n_gn_g}^2 u_{f,g}(\cdot)$ is constant on $\pa\Om_g$. In other words, $\Phi(f,g,\sg_c)=0$ if and only if the pair $(D_f,\Om_g)$ is a solution to the overdetermined problem \eqref{eq}--\eqref{oc}. Indeed, applying the well-known decomposition formula for the Laplace operator 
\begin{equation}\label{decomposition}
    \De \varphi=\pa_{nn}^2 \varphi+H\pa_n \varphi+ \De_\tau \varphi \quad \text{on }\pa\om \quad (\om\in\cC^2,\ \varphi\in\cC^2(\ol \om))
\end{equation}
to the function $u_{f,g}$ yields that the product $H\pa_n u_{f,g}$ is constant on $\pa\Om_g$. Therefore, if $H$ never vanishes on $\pa\Om_g$ (which holds true if $g$ is small enough, by continuity), then we recover the overdetermined condition \eqref{oc}.
\end{remark}
In the rest of this section, we will fix $\sg_c\notin\Si$ and apply Theorem \ref{ift} to the map $\Psi(\cdot,\cdot):=\Phi(\cdot,\cdot,\sg_c)$.

\subsection{Computing the derivative of $\Psi$}
The differentiability of the map $(f,g)\mapsto\Phi(f,g,\sg_c)$ derives from that of the function $u_{f,g}$ and its spatial derivatives up to second order (indeed, notice that $n_g=-\gr u_{f,g}/|\gr u_{f,g}|$ on $\pa\Om_g$, because $\pa\Om_g$ is a level set of $u_{f,g}$ by construction). In turn, the Fr\'echet differentiability of both $u_{f,g}$ and its spatial derivatives can be proved in a standard way, by following the proof of \cite[Theorem 5.3.2, pages 206--207]{HP2005} with the help of the regularity theory for elliptic operators with piecewise constant coefficients. In particular, the H\"older continuity of the first and second derivatives of the function $u_{f,g}$ up to the interface $\pa D_f$, which is stated in \cite[Theorem 16.2, page 222]{LU}, is obtained by
flattening the interface with a diffeomorphism of class $\cC^{2,\al}$ as in \cite[Chapter 4, Section
16, pages 205--223]{LU} or in \cite[Appendix, pages 894--900]{DEF} and by using the classical regularity
theory for linear elliptic partial differential equations (\cite{LU, Gi, ACM}).

Now, for fixed $g_0\in\cG$ and small enough $\ve>0$, consider the map 
\begin{equation}\label{a map, u_t}
(-\ve,\ve)\ni t\mapsto u_t:=u_{0,tg_0}.
\end{equation}
For any given point $x\in\Om_0$, notice that $x\in\Om_{tg_0}$ if $t>0$ is sufficiently small. Therefore it makes sense to consider the following limit:
\begin{equation*}
    \lim_{t\to 0}\frac{u_t(x)-u(x)}{t}.
\end{equation*}
When well defined, the quantity above will be referred to as the \emph{shape derivative} of $u$ and will be denoted by $u'(x)$. 
The following lemma gives a characterization of the shape derivative of $u$ and it will be crucial for the upcoming computations. For a proof, we refer the interested reader to \cite[Proposition 3.1]{cava2018} (see also \cite[Theorem 3.21]{cavaphd}).
\begin{lemma}
Let $g_0\in\cG$ be fixed, and let the map $(-\ve,\ve)\ni t\mapsto u_t$ be as above. Then the shape derivative $u'$ of $u$ is well defined at all points $x\in\Om_0$. Moreover, $u'$ can be characterized as the solution of the following boundary value problem:
  \begin{equation}\label{u'}
        \begin{cases}
        -\dv(\sg\gr u')=0 \quad \text{in }\Om_0,\\
        u'=g_0/N\quad \text{on }\pa\Om_0.
        \end{cases}
    \end{equation}
\end{lemma}
\begin{theorem}\label{Psi'}
The map $\Phi(\cdot,\cdot,\sg_c):\cF\times\cG\to\cH$ is of class $\cC^\infty$ in a neighborhood of $(0,0)$ and its partial Fr\'echet derivative with respect to the second variable defines a continuous linear mapping from $\cG$ to $\cH$ defined by the formula
\begin{equation}\label{pa_y psi}
    \pa_g \Psi(0,0,\sg_c)[g_0]= \pa_{nn}^2 u' \quad \text{for }g_0\in\cG,
\end{equation}
where $u'$ is the solution to \eqref{u'}.
\end{theorem}
The proof of Theorem \ref{Psi'} will be given later, as some preliminary work is required.
From now on, we fix an element $g_0\in\cG$, set $f=0$ and, in order to simplify notations, write $\Om_t, u_t, n_t$ in place of $\Om_{tf}, u_{tf,0}, n_{tg}$. Moreover, for $x\in\Om_0$, we set 
\begin{equation}\label{b A}
b_t(x)=  \gr u_t(x+tg_0(x)n(x))\quad\text{and}\quad A_t(x)= D^2 u_t(x+tg_0(x)n(x)).
\end{equation}
Whenever confusion does not arise, we will omit the subscript $0$ and write $u$ for $u_0$, $n$ for $n_0$, $b$ and $A$ for $b_0$ and $A_0$ and so on. 
The following lemma contains all the ingredients needed for the proof of Theorem \ref{Psi'}. 
\begin{lemma}
The functions $b_t$ and $A_t$ defined in \eqref{b A} are differentiable with respect to the variable $t$ in a neighborhood of $t=0$. Moreover, the following hold true
\begin{equation}\label{b'A'}
\begin{aligned}
b=-\frac{1}{N} n, \quad A= -\frac{1}{N} I,\quad
\dot b:=\restr{\frac{d}{dt}}{t=0}b_t=\gr u'-\frac{g_0}{ N}  n,\quad \dot A:=\restr{\frac{d}{dt}}{t=0} A_t= D^2 u'.
\end{aligned}
\end{equation}
\end{lemma}
\begin{proof}
The first two identities are an immediate consequence of the explicit expression in \eqref{u}.
The differentiability of $b_t$ and $A_t$ with respect to $t$ ensues by standard results concerning shape derivatives (for instance, see \cite{HP2005} or \cite{SG}). The remaining identities in \eqref{b'A'} can be proved as follows. 
\begin{equation*}
\dot b(x)= \restr{\frac{d}{dt}}{t=0}\gr u_t (x+tg_0(x)n(x))= \gr u'(x) + g_0(x) D^2 u(x) n(x)= \gr u'(x)-\frac{g_0(x)}{N} n(x),   
\end{equation*}
where the last equality ensues from \eqref{u}.
Analogously, for the last identity we get
\begin{equation*}
\dot A(x)= \restr{\frac{d}{dt}}{t=0}D^2 u_t (x+tg_0(x)n(x))= D^2u'(x)+g_0(x) D^3u(x)n(x)=D^2u'(x),
\end{equation*}
where, in the last equality, we used the fact that $D^3u=0$ because $u$ is quadratic in a neighborhood of $\pa\Om_0$ (again, by \eqref{u}).
\end{proof}

\begin{proof}[Proof of Theorem \ref{Psi'}]
Since $\Phi$ is Fr\'echet differentiable, we can compute $\pa_g \Phi(0,0,\sg_c)$ as a G\^ateaux derivative:
\begin{equation}\label{psi by c}
    \pa_g \Phi(0,0,\sg_c)[g_0]= \lim_{t\to 0^+}\frac{\Phi(0,tg_0,\sg_c)-\Phi(0,0,\sg_c)}{t}\quad \text{for }g_0\in\cG.
\end{equation}
First of all, notice that, for fixed $g_0\in\cG$, the map $\Phi(0,tg_0,\sg_c)$ can be expressed by means of the auxiliary functions $b_t$ and $A_t$ defined in \eqref{b A}:
\begin{equation*}
    \Phi(0,tg_0,\sg_c)=  c_t - \frac{1}{|\pa\Om_t|}\int_{\pa\Om_t} c_t,\qquad\text{where }c_t=\frac{b_t \cdot (A_t b_t)}{|b_t|^2}.
\end{equation*}
Let us first focus on the derivative of $c_t$. By the quotient rule, we have
\begin{equation*}
    \restr{\frac{d}{dt}}{t=0}c_t=\restr{\frac{d}{dt}}{t=0}\left(\frac{b_t \cdot (A_t b_t)}{|b_t|^2}\right) = \frac{|b|^2(2\dot b\cdot A b + b\cdot \dot A b)-2b\cdot \dot b (b\cdot A b)}{|b|^4}
\end{equation*}
Substituting the expressions for $b$, $A$, $\dot b$ and $\dot A$ from \eqref{b'A'} into the expression above yields
\begin{equation}\label{c'}
    \restr{\frac{d}{dt}}{t=0}c_t= -{2}{\left(\gr u'-\frac{g_0}{N}n\right)}\cdot n + \pa_{nn}^2u'+{2}\left(\gr u'-\frac{g_0}{N}n\right)\cdot n=\pa_{nn}^2u'.
\end{equation}
Now, let us compute the derivative of the remaining term in \eqref{psi by c} with \eqref{c'} at hand. 
By employing the use of the well-known formula for the shape derivative of the perimeter (see \cite[Corollary 5.4.16 and underneath remarks, page 224]{HP2005})
\begin{equation*}
    \restr{\frac{d}{dt}}{t=0}|\pa\Om_t|=\int_{\pa\Om_0}H g_0,
\end{equation*}
we obtain
\begin{equation}\label{c''}
    \restr{\frac{d}{dt}}{t=0}\left( \frac{1}{|\pa \Om_t|}\int_{\pa\Om_t} c_t\right) = \frac{|\pa\Om_0|\int_{\pa\Om_0} \pa_{nn}^2u'-\int_{\pa\Om_0} H g_0 \int_{\pa\Om_0}\pa_{nn}^2u}{|\pa\Om_0|^2}=\frac{1}{|\pa\Om_0|}\int_{\pa\Om_0}\pa_{nn}^2u', 
\end{equation}
where in the last equality we used the fact that $H$ is constant on $\pa\Om_0$ and that $g_0\in\cG$ has vanishing average by hypothesis. 
The claim of the theorem follows if we manage to show that the integral of $\pa_{nn}^2u'$ over $\pa\Om_0$ vanishes. In order to show this, let us apply \eqref{decomposition} to $u'$. We get
\begin{equation}\label{c'''}
    \int_{\pa\Om_0}\pa_{nn}^2u'= -H \int_{\pa\Om_0}\pa_n u' + \int_{\pa\Om_0}\De_\tau g_0= -H \int_{\pa\Om_0}\pa_n u'. 
\end{equation}
In a similar way to what we did in the proof of Theorem I, by the divergence theorem we conclude that 
\begin{equation}\label{c''''}
    0=\sg_c\int_{D_0}\De u'+\int_{\Om_0\setminus\ol{D_0}} \De u' = \int_{\pa D_0} [\sg\pa_n u'] +\int_{\pa\Om_0}\pa_n{u'}=\int_{\pa\Om_0}\pa_n u'.
\end{equation}
The claim follows by combining \eqref{c'},\eqref{c''}, \eqref{c'''} and \eqref{c''''}.
\end{proof}

\subsection{Applying the implicit function theorem}
In order to apply Theorem \ref{ift}, we will need the following explicit representation for $u'$ as a spherical harmonic expansion. 
Let $\{Y_{k,i}\}_{k,i}$ ($k\in \{0,1,\dots\}$, $i\in\{1,2,\dots, d_k\}$) denote a maximal family of linearly independent solutions to the eigenvalue problem
\begin{equation*}
-\De_{\tau} Y_{k,i}=\lambda_k Y_{k,i} \quad\textrm{ on }\SS^{N-1},
\end{equation*}
with $k$-th eigenvalue $\lambda_k=k(N+k-2)$ of multiplicity $d_k$ and normalized in such a way that $\norm{Y_{k,i}}_{L^2(\SS^{N-1})}=1$. Here $\De_\tau$ stands for the Laplace--Beltrami operator on the unit sphere $\SS^{N-1}$. Such functions, usually referred to as \emph{spherical harmonics} in the literature, form a complete orthonormal system of $L^2(\SS^{N-1})$. Notice that the eigenspace corresponding to the eigenvalue $\la_0=0$ is the 1-dimensional space of constant functions on $\SS^{N-1}$. Moreover, notice that such constant function term does not appear in the expansion of a function of zero average on $\SS^{N-1}$. 
We refer to \cite[Proposition 3.2]{cava2018} for a proof of the following result. 
\begin{lemma}\label{lemma u' expansion}
Assume that, for some real coefficients $\alpha_{k,i}$, the following expansion holds true in $L^2(\SS^{N-1})$:
\begin{equation}\label{h_in h_out exp}
%f(R\theta)=\sum_{k=1}^\infty\sum_{i=1}^{d_k}\al_{k,i}^- Y_{k,i}(\theta), \quad 
g_0(\theta)=\sum_{k=1}^\infty\sum_{i=1}^{d_k}\al_{k,i} Y_{k,i}(\theta).
\end{equation} 
Then, the function $u'$, solution to \eqref{u'}, admits the following explicit expression for $\theta\in\mathbb{S}^{N-1}$ and $r\in [0,1]$:
\begin{equation}\label{u' li seme}
 u'(r\theta)=\begin{cases}\displaystyle
\sum_{k=1}^\infty\sum_{i=1}^{d_k}\al_{k,i} B_k r^k Y_{k,i}(\theta) \quad &\text{ for } r\in[0,R],\\
\displaystyle
\sum_{k=1}^\infty\sum_{i=1}^{d_k}\al_{k,i}\left(C_k r^{2-N-k}+D_k r^k\right) Y_{k,i}(\theta)\quad&\text{ for } r\in(R,1],
\end{cases}
\end{equation}
where $R\in (0,1)$ and the coefficients $B_k$, $C_k$ and $D_k$ are defined as follows:
\begin{equation*}
%\begin{aligned}
%&B_k^- = \frac{1-\sg_c}{\sg_c}R^{-k+1}\left( (N-2+k)R^{2-N-2k} +k\right)/F, \quad 
%C_k^- = (\sg_c-1)k R^{-k+1}/F, \quad D_k^-=-C_k^-, \\
B_k = (N-2+2k)R^{2-N-2k}/F, \quad C_k = (1-\sg_c)k/F,\quad D_k = (N-2+k+k\sg_c)R^{2-N-2k}/F,
%\end{aligned}
\end{equation*}
and the common denominator $F= N(N-2+k+k\sg_c)R^{2-N-2k}+k N (1-\sg_c)$.

When $D=\emptyset$, that is $R=0$ (or $\sg_c=1$), then the above simplifies to 
\begin{equation*}\label{lemma3.5 onephase}
u'(r\theta)= \sum_{k=1}^\infty\sum_{i=1}^{d_k}\frac{r^k}{N} \al_{k,i} Y_{k,i}(\theta).    
\end{equation*}
\end{lemma}
\begin{remark}\label{F>0}
The quantity $F=N(N-2+k+k\sg_c)R^{2-N-2k}+k N (1-\sg_c)$ is strictly positive if $R\in(0,1)$ and $k\ge 1$. Indeed we have
\begin{equation*}
F> N(N-2+k+k\sg_c)+k N(1-\sg_c)=N^2+2N(k-1)\ge N^2>0.    
\end{equation*}
\end{remark}
Now, with Theorem \ref{Psi'} and expansion \eqref{u' li seme} at hand, it is easy to check that the partial Fr\'echet derivative $\pa_g\Phi(0,0,\sg_c)$ yields the following map from $\cG$ into $\cH$ defined by:
\begin{equation}\label{longmapsto}
    \sum_{k=1}^\infty\sum_{i=1}^{d_k} \al_{k,i}Y_{k,i}\longmapsto \sum_{k=1}^\infty\sum_{i=1}^{d_k} \al_{k,i}
    \be_k Y_{k,i}, \quad \text{where }
\end{equation}
%where 
\begin{equation*}
    \be_k= \be_k(\sg_c)= k\ \frac{(2-N-k)(1-N-k)(1-\sg_c)+(k-1)(N-2+k+k\sg_c)R^{2-N-2k}}{k N (1-\sg_c)+
    N(N-2+k+k\sg_c)R^{2-N-2k}}. 
\end{equation*}
We are now ready to apply the implicit function theorem to the mapping $\Phi(\cdot,\cdot,\sg_c)$.
\begin{proof}[Proof of Theorem \ref{thm II}]
Since the Fr\'echet differentiability of $\Phi$ has already been dealt with in Theorem \ref{Psi'}, in order to apply Theorem \ref{ift} of page \pageref{ift} to $\Phi(\cdot,\cdot,\sg_c)$, we just need to ensure that the mapping defined by \eqref{pa_y psi} (or, equivalently, \eqref{longmapsto}) is a bounded linear transformation from $\cG$ to $\cH$. Linearity and boundedness ensue from the properties of the boundary value problem \eqref{u'}. We are left to show that $\pa_g\Phi(0,0,\sg_c):\cG\to\cH$ is a bijection. Injectivity is immediate, once one realizes that, for $k\ge 1$, the coefficient $\be_k$ in \eqref{longmapsto} vanishes if and only if $\sg_c=s_k$ (in retrospective, we can say that $s_k$ was defined in order to have this property). This implies that the map $\pa_g\Phi(0,0,\sg_c)$ is injective as long as $\sg_c\notin\Si$. Let us now show surjectivity. Take an arbitrary function $h_0\in\cH$. Since, in particular, $h_0$ is continuous on $\pa\Om_0$, it admits a spherical harmonic expansion, say 
\begin{equation*}
    h_0=\sum_{k=1}^\infty\sum_{i=0}^{d_k} \ga_{k,i} Y_{k,i}. 
\end{equation*}
Set now 
\begin{equation}\label{g}
    g_0=\sum_{k=1}^\infty\sum_{i=0}^{d_k} \frac{\ga_{k,i}}{\be_k} Y_{k,i}. 
\end{equation}
First of all, notice that, since the sequence $1/\be_k$ is bounded, the function $g_0$ above is a well defined element of $L^2(\pa\Om_0)$. Moreover, the integral of $g_0$ over $\pa\Om_0$ vanishes because the summation in \eqref{g} starts from $k=1$. 
Finally, if we let $\mathcal L$ denote the continuous extension to $L^2(\pa\Om_0)\to L^2(\pa\Om_0)$ of the map defined by \eqref{longmapsto}, it is clear that $g_0=\mathcal{L}^{-1}(h_0)$ by construction. Therefore, in order to prove the surjectivity of the original map $\pa_g\Phi(0,0,\sg_c)$, we just need to show that the function $g_0$, defined in \eqref{g}, is of class $\cC^{2,\al}$ whenever $h_0\in \cC^\al$. To this end, we will employ the use of various facts from classical regularity theory. First of all, we recall that functions in the Sobolev space $H^s(\pa\Om_0)$ can be characterized by the decay of the coefficients of their spherical harmonic expansion as follows:
\begin{equation*}
    \sum_{k=1}^\infty\sum_{i=0}^{d_k}(1+k^2)^s \al_{k,i}^2 < \infty \iff \sum_{k=1}^\infty\sum_{i=0}^{d_k} \al_{k,i}^2 Y_{k,i}\in H^s(\pa\Om_0).
\end{equation*}
Since $h_0\in\cC^\al(\pa\Om_0)\subset L^2(\pa\Om_0)$, the asymptotic behavior of the coefficients $\be_k$ given in \eqref{longmapsto} yields that $g_0\in H^2(\pa\Om_0)$. Now, if we define $u'$ to be the solution to \eqref{u'} whose $g_0$ in the boundary condition is given by \eqref{g}, then we also obtain that $   \restr{\pa_{nn}^2 u'}{\pa\Om_0}=h_0$ and $\restr{\pa_n u'}{\pa\Om_0}\in H^1(\pa\Om_0)$. 
The next step employs the use of the decomposition formula \eqref{decomposition}, which still holds true by a density argument. We get
\begin{equation*}
    \pa_{nn}^2 u' + H \pa_n u' + 1/N \ \De_\tau g_0= 0 \quad\text{on }\pa\Om_0.
\end{equation*}
That is,
\begin{equation}\label{deco2}
h_0 + (N-1)\pa_n u' + 1/N \ \De_\tau g_0= 0 \quad\text{on }\pa\Om_0.
\end{equation}
The identity above implies that, in particular, $\De_\tau g_0$ belongs to $L^p(\pa\Om_0)$ for $p=\frac{2N}{N-1}$. Therefore, by the standard $L^p$ theory for the Laplace equation on manifolds, we get that $g_0\in W^{2,p}(\pa\Om_0)$. Then, by the trace theorem for general Sobolev spaces, $u'$ is of class $W^{2+\frac{1}{p},p}$ in a interior tubular neighborhood of $\pa\Om_0$, and hence $\restr{\pa_n u'}{\pa\Om_0}\in W^{1,p}(\pa\Om_0)$. This fact, together with \eqref{deco2}, implies that $\restr{\pa_n u'}{\pa\Om_0}$ belongs to an $L^q$ space with a higher integration exponent $q$. Iterating this process one gets that $\restr{\pa_n u'}{\pa\Om_0}\in W^{1,p}(\pa\Om_0)$ for all $p>1$. Thus, by Morrey's inequality, $\pa_n u'$ also belongs to $\cC^{\al}(\pa\Om_0)$. Going back to the identity \eqref{deco2}, this implies that also $\De_\tau g_0\in \cC^{\al}(\pa\Om_0)$ and thus, by the Schauder theory for the Laplace operator on manifolds, we finally obtain that $g_0\in\cC^{2,\al}(\pa\Om_0)$, as claimed. This concludes the proof of the invertibility of the map $\pa_g\Phi(0,0,\sg_c):\cG\to\cH$ and thus that of Theorem II.      
\end{proof}

\section{Symmetry breaking bifurcation at $\sg_c\in\Si$}
In this section we will show the local behavior of nontrivial solutions to \eqref{eq}--\eqref{oc} near the trivial solution $(D_0,\Om_0)$ when $\sg_c\in\Si$. To this end, we will employ the use of the following version of the Crandall--Rabinowitz theorem (that is equivalent to the one stated in \cite{CR}). Although, nowadays, the Crandall--Rabinowitz theorem can be regarded as a staple of bifurcation theory for partial differential equations, to the best of my knowledge its applications to overdetermined problems are not so well-known (see \cite{Oka, FR, EM, KS, CYisaac} for some literature). 

\begin{thm}[Crandall--Rabinowitz theorem \cite{CR}]\label{Crandall--Rabinowitz theorem}
Let $X$, $Y$ be real Banach spaces and let $U\subset X$ and $\La\subset \RR$ be open sets, such that $0\in U$.
Let $\Psi\in\cC^p(U\times\La;Y)$ ($p\ge3$) and assume that there exist $\la_0\in\La$ and $x_0\in\cX$ such that 
\begin{enumerate}[label=(\roman*)]
    \item $\Psi(0,\la) = 0$ for all $\la\in\La$;
    \item ${\rm Ker}\ \pa_x \Psi(0,\la_0)$ is a 1-dimensional subspace of $X$ spanned by $x_0$; 
    \item ${\rm Im}\ \pa_x \Psi(0,\la_0)$ is a closed co-dimension 1 subspace of $Y$;
    \item  $\pa_\la\pa_x \Psi(0,\la_0)[x_0]\notin {\rm Im}\ \pa_x \Psi(0,\la_0)$.
    \end{enumerate}
Then $(0, \la_0)$ is a bifurcation point of the equation $\Psi(x,\la)=0$ in the following sense. 
In a neighborhood of $(0, \la_0)\in\cX\times\La$, the set of solutions of $\Psi(x,\la) = 0$ consists of two $\cC^{p-2}$-smooth curves $\Gamma_1$ and $\Gamma_2$ which intersect only at the point $(0,\la_0)$. $\Gamma_1$ is the curve $\{(0,\la)\,:\,\la\in\La\}$ and $\Gamma_2$ can be parametrized as follows, for small $\ve>0$: 
\begin{equation*}
    (-\ve,\ve)\ni t \mapsto\left(x(t),\la(t)\right)\in\cU\times\La,\text{ such that }\left(x(0),\la(0)\right)=(0,\la_0), \quad x'(0)=x_0. 
\end{equation*}
%\begin{equation*}
%\Gamma_2: \left( x(\varepsilon),\gamma(\varepsilon) \right), \quad \varepsilon:\text{small}, \quad \left( x(0),\gamma(0) \right) = (0,\gamma_0), \quad x'(0)=x_0.  
%\end{equation*}
\end{thm}

In what follows, we will try to apply Theorem \ref{Crandall--Rabinowitz theorem} to study the local behavior of the map (which is different from the one that was used in the previous section) 
\begin{equation}
    \Psi(g,s):= \Phi(0,g,s), \quad (g\in\cG, s>0),
\end{equation}
around the bifurcation points $\la_0=s_k$. 
Unfortunately, we cannot directly apply Theorem ~\ref{Crandall--Rabinowitz theorem} in this setting because ${\rm Ker}\ \pa_g\Psi(0,0,s_k)$ is not a $1$-dimensional vector space. In order to circumvent this problem, we will consider the restriction of $\Phi(0,\cdot,s_k)$ to some particular invariant subspace of $\cG$. 

Here we recall the definition of invariant subspace. Let $G$ be a subgroup of the orthogonal group $O(N)$. We will say that
a set $\om\subset\rn$ is $G$-invariant if $\ga(\om)=\om$ for all $\ga\in G$. Moreover, a real-valued function defined on a $G$-invariant domain $\om$ is said to be $G$-invariant if 
\begin{equation*}
    \varphi=\varphi\circ\ga\quad \text{for all }\ga\in G. 
\end{equation*}
Suppose that $\sg_c=s_k\in\Si$ for some $k$, and let $\cY_k$ denote the $k$-th eigenspace of $-\De_\tau$, that is, the subspace of $\cC^{2,\al}(\pa\Om_0)$ spanned by $\{Y_{k,1},\dots, Y_{k,d_k}\}$.
Moreover, let $G^\ast$ be a subgroup of $O(N)$ such that the invariant subspace
\begin{equation}\label{Y*}
    \cY_k^\ast = \left\{ Y_k\in \cY_k \;:\; Y_k \text{ is $G^\ast$-invariant}  \right\} \text{ is $1$-dimensional}
\end{equation}
(see the Appendix, for a proof that $G^\ast={\rm Id}\times O(N-1)$ satisfies \eqref{Y*} for all $k\in\NN$).  
Let us now define the following two invariant spaces:
\begin{equation*}
    \cG^\ast=\left\{g\in\cG\;:\; g \text{ is $G^\ast$-invariant}\right\},\quad
    \cH^\ast=\left\{h\in\cH\;:\; h \text{ is $G^\ast$-invariant}\right\},
\end{equation*}
and let $\Phi^\ast(0,\cdot,\cdot)$ denote the restriction of $\Phi(0,\cdot,\cdot)$ to $\cG^\ast\times (0,\infty)$. We claim that such a $\Phi^\ast(0,\cdot,\cdot)$ defines a mapping $\cG^\ast\times (0,\infty)\to\cH^\ast$. Indeed, for all $g\in\cG^\ast$, we have that $(D_0,\Om_g)$ is a pair of $G^\ast$-invariant domains. As a consequence, by the unique solvability of the boundary value problem \eqref{eq}--\eqref{bc}, the function $u_{0,g}$ is $G^\ast$-invariant as well and, therefore, so is $\Phi(0,g,\sg_c)$, as claimed. 

\begin{proof}[Proof of Theorem \ref{thm III}]
Let $\sg_c=s_k\in\Si$ and let $Y_k$ be an element in $\cG^\ast$ that spans the 1- dimensional subspace $\cY_k^\ast$ defined in \eqref{Y*}.
In order to prove Theorem \ref{thm III} we will just need to check conditions $(i)$--$(iv)$ of Theorem \ref{Crandall--Rabinowitz theorem} with respect to the map $\Phi^\ast(0,\cdot,\cdot):\cG^\ast\times(0,\infty)\to\cH^\ast$ at the bifurcation point $(0,s_k)\in\cG^\ast\times(0,\infty)$.
Condition $(i)$ is clearly true, as it is equivalent to saying that the trivial solution $(D_0,\Om_0)$ is a solution of the overdetermined problem \eqref{eq}--\eqref{oc} for all values of $\sg_c>0$. Conditions $(ii)$ and $(iii)$ are also true because, by construction, the kernel ${\rm Ker}\ \pa_g\Phi^\ast(0,0,s_k)=\cY_k^\ast$ is a 1-dimensional subspace of $\cH^\ast$ spanned by $Y_k$, and, similarly, the image ${\rm Im}\ \pa_g\Phi^\ast(0,0,s_k)= \cH^\ast/\cY_k^\ast$ is a closed subspace of $\cH^\ast$ of codimension $1$. In what follows we will show that condition $(iv)$ also holds true. 
By \eqref{longmapsto},
\begin{equation*}
    \pa_g \Phi^\ast(0,0,s) [Y_k]=\be_k(s) Y_k \quad \text{for }s>0,
\end{equation*}
which in turn implies
\begin{equation*}
    \pa_s \pa_g \Phi^\ast(0,0,s_k)[Y_k]= \big(\pa_s\be_k(s_k)\big) Y_k.
\end{equation*}
In other words, in order for $(iv)$ to hold true, we need to show that the derivative $\pa_s\be_k(s_k)$ does not vanish.  A direct computation of $\pa_s\be_k(s_k)$ is possible by recalling the definition of $\be_k(s)$ in \eqref{longmapsto} and the fact that $\be_k(s_k)=0$ by the defining property of $s_k$. We obtain 
\begin{equation*}
\pa_s\be_k(s_k)= k\frac{-(k+N-2)(k+N-1)+k(k-1)R^{2-N-2k}}{k N (1-\sg_c)+
    N(N-2+k+k\sg_c)R^{2-N-2k}}.    
\end{equation*}
By combining \eqref{nondegeneracy} and Remark \ref{F>0} we get that $\pa_s\be_k(s_k)<0$ and, in particular, the function $\pa_s \pa_g \Phi^\ast(0,0,s_k)[Y_k]$ does not belong to the image ${\rm Im}\ \pa_g\Phi^\ast(0,0,s_k)$ as claimed. This concludes the proof of Theorem \ref{thm III}.
\end{proof}
\begin{remark}
We claim that all nontrivial solutions $(D_0,\Om_{g(\varepsilon)})$ given by Theorem \ref{thm III} share the same symmetries of the element $x_0\in X^*$, defined such that ${\rm Ker} \, \pa_x \Psi^*(0,0)={\rm span}\{Y_k\}$. To this end, take a symmetry group $G\subset O(N)$ such that the function $Y_k$ is $G$-invariant. Now, consider the further restriction $\Psi^{**}$ of $\Psi^*$ to the subspace $X^{**}$ of all $G$-invariant functions in $X^*$. Notice that, since $Y_k$ is $G$-invariant by hypothesis, we have ${\rm Ker} \, \pa_g\Psi^*(0,0)={\rm Ker} \, \pa_g \Psi^{**}(0,0)={\rm span}\{Y_k\}$. Another application of the Crandall--Rabinowitz theorem to $\Psi^{**}$ yields that $g(\varepsilon)$ is also $G$-invariant. The claim follows by the arbitrariness of $G$.
\end{remark}

\section{More about the one phase case ($D=\emptyset$)}

Let $\Om_0$ denote the unit ball centered at the origin and let $\Om_g$ be the perturbation of $\Om_0$ by a function $g\in\cG$ as defined in \eqref{perturbed domains}. We know that, when $D=\emptyset$, $\Om_0$ is a solution of overdetermined problem \eqref{eq}--\eqref{oc} for $d=\frac{N-1}{N}$.
In what follows, we will show that, unlike the two phase case $D\ne\emptyset$, trivial solutions are isolate solutions for \eqref{eq}--\eqref{oc} when $D$ is empty.
%Suppose that, for some small $\ve>0$,
%\begin{equation}\label{germs of solutions}
%    \begin{aligned}
%\text{there exist two $\cC^1$ curves $t\mapsto g(t)\in\cG$ and $t\mapsto d(t)\in\RR$ such that}\\
%\text{$g(0)=0$ and the set $\Om_{g(t)}$ is a solution of \eqref{eq}--\eqref{oc} for $d=d(t)$}.
%    \end{aligned}
%\end{equation}
%If \eqref{germs of solutions} happens only when the sets $\Om_{g(t)}$ are translations of $\Om_0$, that is there exists a curve $t\mapsto x_0(t)\in\rn$ such that 
%\begin{equation}\label{translations}
%    \Om_{g(t)}= x_0(t)+\Om_0 \quad \text{for all $t\in (-\ve,\ve)$},
%\end{equation}
\begin{definition}[Isolate solution]\label{isolate sol}
We say that the trivial solution $\Om_0$ is an \emph{isolate solution of} \eqref{eq}--\eqref{oc} for $D=\emptyset$ if the following holds. 

There exists some $\eta>0$ such that, for all elements $g_0\in\cG$ with $\norm{g_0}_{\cC^{2,\al}}<\eta$, the set $\Om_{g_0}$ is a solution of \eqref{eq}--\eqref{oc} for some $d=d(g_0)$ if and only if 
\begin{equation*}
    \Om_{g_0}=x_0+\Om_0 \quad \text{for some }x_0=x_0(g_0)\in\rn.
\end{equation*}
In other words, $\Om_0$ is said to be a trivial solutions, if the only solutions of \eqref{eq}--\eqref{oc} in a neighborhood of $\Om_0$ are precisely translations of $\Om_0$. 
\end{definition}

In order to prove Theorem \ref{thm IV}, we will make use of the following construction. 
Let $\cY_1$ denote the eigenspace of spherical harmonics corresponding to the first nonzero eigenvalue $\la_1=N-1$. $\cY_1$ is an $N$-dimensional space of analytic functions on the unit sphere $\pa\Om_0$. Without loss of generality, we can write $\cY_1={\rm span}\left\{Y_{1,1},\dots,Y_{1,N}\right\}$, where $Y_{1,i}$ ($i=1,\dots, N$) are the functions defined by
\begin{equation}\label{Y_1,i}
    Y_{1,i}(x)=\sqrt{\frac{N}{|\pa\Om_0|}}\ x_i, \quad \text{for }%i=1,\dots,N \, \text{and }
    x=(x_1,\dots,x_N)
    \in\pa\Om_0.
\end{equation}
Moreover, let $\Pi_1:\cC^\al(\pa\Om_0)\to \cY_1$ denote the projection operator onto the eigenspace $\cY_1$ and set $Q={\rm Id}-\Pi_1$.
Consider now the following map: 
\begin{equation*}
\begin{aligned}
    \Psi:\; \rn\times\cG &\longrightarrow \rn\times(\cH/\cY_1)\\
    (y,g)&\longmapsto \left({\rm Bar}\ \Om_g-y, Q\Phi(g)\right),
\end{aligned}
\end{equation*}
where ${\rm Bar}\ A$ denotes the barycenter of the set $A$, that is, the point $\int_A x$, and by a slight abuse of notation, $\Phi(g)$ denotes the one-phase version of \eqref{Phi} (in other words, $\Phi(g):=\Phi(0,g,1)$).
The following lemma plays a key role in the proof of Theorem \ref{thm IV}.
\begin{lemma}\label{uniqueness one phase}
There exists a small positive real number $\ve>0$ and a unique map $g:B_\ve(0)\to\cG$ such that the set of solutions of the equation $\Psi(y,g)=0$ around $(0,0)\in\rn\times\cG$ can be locally expressed as
\begin{equation*}
\left\{ (y,g)\in B_\ve(0)\times\cG\;:\; \Psi(y,g)=0  \right\}=\left\{ \left(y,g(y)\right)\;:\; y\in B_\ve(0)\right\}.
\end{equation*}
\end{lemma}
\begin{proof}
%Let $g(\cdot)$ and $d(\cdot)$ be two functions satisfying \eqref{a brach of zeros}. 
 %Set $v(t)=\Pi_1g(t)$. We know that $v(t)$ can be ``completed to a translation". That is, there exists a $\cC^1$-map 
%\begin{equation*}
%    t\longmapsto \widetilde w (t)\in \cG/\cY_1
%\end{equation*}
%such that the function $\widetilde g(t):=v(t)+\widetilde w(t)$ induces a translation of $\Om_0$ in the sense of \eqref{translations}.
%To put it differently, we have two functions, $g(\cdot)$ and $\widetilde g(\cdot)$ that satisfy \eqref{a brach of zeros}, one of which induces a translation of $\Om_0$. In what follows we will employ the use of the implicit function theorem (Theorem \ref{ift}) to show that indeed $g=\widetilde g$, and thus $\Om_0$ is an isolated solution as claimed.  
%Set $Q={\rm Id}-\Pi_1$ and consider the following map: 
%\begin{equation*}
%\begin{aligned}
%    \Psi:\ \RR\ \times&\  \cG/\cY_1\longrightarrow \cH/\cY_1\\
%    \quad &(t,w)\longmapsto  Q\Phi(0,v(t)+w,1).
%\end{aligned}
%\end{equation*}
%Notice that, for all pairs $(t,w)$ such that $\Psi(t,w)$ is well defined, if $\Om_{v(t)+w}$ solves \eqref{eq}--\eqref{oc} then we must have $\Psi(t,w)=0$ (notice, that, because we are employing the projection operator $Q$, nothing can be said about the reverse implication in general). 
We will apply the implicit function theorem to the map $\Psi$ above. Indeed, $\Psi$ is a well-defined $\cC^1$-mapping in a neighborhood of $(0,0)\in \rn\times\cG$ because both the barycenter function ${\rm Bar}(\cdot)$ and $\Phi(\cdot)$ are. Moreover, by construction, we have $\Psi(0,0)=(0,0)$.

In what follows, we will give the explicit formula for the Fr\'echet derivative $\pa_g\Psi(0,0)$. First of all, we recall the explicit formula for the Fr\'echet derivative of the barycenter function:
\begin{equation}\label{deri bari}
    \restr{\left(\pa_g {\rm Bar}\ \Om_g\right)}{g=0}[g_0]=\int_{\pa\Om_0}g_0 \ n \quad \text{for }g_0\in\cG.
\end{equation}
The expression in \eqref{deri bari} can be obtained by applying the Hadamard formula to the real-valued functions $g\mapsto \int_{\Om_g}x_i$ for $i=1,\dots,N$ (see, for instance, the second example in \cite[Subsection 5.9.3]{HP2005}).
Now, combining the formula for the shape derivative of the barycenter \eqref{deri bari} and Theorem \ref{Psi'} yields
\begin{equation*}
\pa_g\Psi(0,0)[g_0]= \left(\int_{\pa\Om_0} g_0\ n ,\ Q\left( \pa_{nn}^2\restr{u'}{\pa\Om_0}\right)\right),
\end{equation*}
where $u'$ is the unique solution of \eqref{u'} with $D=\emptyset$.
Now, if we expand $g_0$ as in \eqref{h_in h_out exp}, then by Lemma \ref{lemma u' expansion} and \eqref{Y_1,i} we get 
\begin{equation*}
    \pa_g\Psi(0,0)\left[\sum_{k=1}^\infty\sum_{i=1}^{d_k} \al_{k,i} Y_{k,i}   \right] = \left( 
    \sqrt{\frac{\pa\Om_0}{N}}\begin{pmatrix}
           \al_{1,1} \\
           %\al_{1,2} \\
           \vdots \\
           \al_{1,N}
         \end{pmatrix}, \ 
    \sum_{k=2}^\infty\sum_{i=1}^{d_k}\frac{k (k-1)}{N}\al_{k,i} Y_{k,i}\right).
\end{equation*}
Now, by reasoning along the same lines as in the proof of Theorem \ref{thm II} in section 3, we conclude that there exists a unique map $y\mapsto g(y)\in\cG$ such that %the set $\Om_{v(t)+w(t)}$ solves \eqref{eq}--\eqref{oc} when $\sg_c=1$.
$\Psi(y,g(y))=0$ for $|y|$ sufficiently small.
%Therefore, by uniqueness, 
%\begin{equation*}
%    g(t)=v(t)+w(t)=v(t)+\widetilde w(t)=\widetilde g(t).
%\end{equation*}
%That is, $g(t)$ induces a translation of $\Om_0$, and hence $\Om_0$ is an isolated solution, as claimed.
\end{proof}
We are now ready to give a proof of Theorem~\ref{thm IV}.
\begin{proof}[Proof of Theorem \ref{thm IV}]
First of all, notice that, if $\Om_g$ solves overdetermined problem \eqref{eq}--\eqref{oc}, then $\Psi({\rm Bar}\ \Om_g, g)=0$ (notice also that the converse is not necessarily true in general). In particular, for $|y|$ small, let $\widetilde g(y)$ denote the unique element in $\cG$ such that 
\begin{equation*}
\Om_{\widetilde g(y)}=y+\Om_0. %\quad \text{(which is well defined for $|y|$ small)},    
\end{equation*}
We have that $\Psi(y,\widetilde g(y))=0$ and thus, by Lemma \ref{uniqueness one phase} there exists some $\ve>0$ such that 
\begin{equation}\label{g's}
    g(y)=\widetilde g(y)\quad (\text{if }|y|<\ve.
\end{equation}
Let now $g_0$ be an element of $\cG$ such that the set $\Om_{g_0}$ solves \eqref{eq}--\eqref{oc}, and put $x_0={\rm Bar}\ \Om_{g_0}$. We claim that $\Om_{g_0}=x_0+\Om_0$ if $g_0$ is small enough. 
Indeed, notice that, by construction, $\Psi(x_0,g_0)=0$. Now, suppose that the function $g_0$ is sufficiently small, so that $|x_0|<\ve$. Then, by Lemma \ref{uniqueness one phase}, we obtain $g_0=g(x_0)$. In particular, \eqref{g's} yields that $g(x_0)=\widetilde g(x_0)$. Combining all these and then using the definition of $\widetilde g(\cdot)$, we get
\begin{equation*}
    \Om_{g_0}=\Om_{g(x_0)}=\Om_{\widetilde g(x_0)}=x_0+\Om_0.
\end{equation*}
Since the choice of $g_0\in\cG$ was arbitrary, we conclude that $\Om_0$ is an isolate solution as claimed.
\end{proof}
\begin{remark}
We still do not know whether the only solutions of \eqref{eq}--\eqref{oc} are trivial when $D=\emptyset$. Indeed, Theorem~\ref{thm IV} (as it is stated) leaves open the possibility of solutions of the form $\Om_g$ that suddenly appear for $\norm{g}_{\cC^{2,\al}}$ large or even that of solutions with a more intricate topology corresponding to some nontrivial value $d<\frac{N-1}{N}$.
\end{remark}

\section{Comparison with the two phase Serrin overdetermined problem}

In this section, we will analyze the main similarities and differences between the two phase overdetermined problems \eqref{eq}--\eqref{oc} and \eqref{eq s}--\eqref{oc s}. The latter was first studied by Serrin in the '70s in the case $D=\emptyset$ employing the moving planes method. 

\begin{theorem}[\cite{Se1971}]\label{thm serrin}
Let $D=\emptyset$. Then the overdetermined problem \eqref{eq s}--\eqref{oc s} admits a solution if and only if $\Om$ is a ball.
\end{theorem}
\begin{remark}
Notice that Theorem \ref{thm serrin} is a global theorem, while Theorem \ref{thm IV} is only local. One might wonder whether it is possible to extend Serrin's proof to the overdetermined problem \eqref{eq}--\eqref{oc}. The main difficulty lies in the following observation. The overdetermined condition \eqref{oc} translates to an overdetermination on the second normal derivative and therefore, one cannot rely on the maximum principle (thus, the moving plane method) in any obvious way.  
\end{remark}

A crucial difference between the two overdetermined problems lies in the degrees of freedom given by the constants $d$ and $c$. Indeed, if $(D,\Om)$ solves the two phase overdetermined problem of Serrin type \eqref{eq s}--\eqref{oc s}, then by integration by parts we get $c=|\Om|/|\pa\Om|$. That is, the constant $c$, although independent of the core $D$, is not scaling invariant, and hence different trivial solutions might take different values of $c$. On the other hand, the constant $d$ in \eqref{oc} is dimensionless. As a consequence, all trivial solutions of \eqref{eq}--\eqref{oc} share the same $d=\frac{N-1}{N}$ (and the converse is also true by Theorem \ref{thm I}). 

\begin{remark}
The overdetermined condition of Serrin type \eqref{oc s} arises naturally in the context of shape optimization. Indeed, let $D$ and $\sg_c$ be given and consider the following shape functional:
\begin{equation*}
    E(\Om)=\int_\Om \sg_c|\gr u|^2,
\end{equation*}
where $u$ is the unique solution to \eqref{eq s}--\eqref{bc s}. Indeed, if $\Om$ is a critical shape of the functional $E(\cdot)$ among all domains of a given volume, then $u$ must automatically satisfy condition \eqref{oc s} (this is a consequence of the computations done in \cite{cava2018}). 
To the best of my knowledge, it is still not known whether the overdetermined condition \eqref{oc} also arises as an optimality condition for some sensible shape functional. 
\end{remark} 

It is interesting to note that, the two phase overdetermined problems \eqref{eq}--\eqref{oc} and \eqref{eq s}--\eqref{oc s} show a nearly identical local behavior. Indeed, let $(D_0,\Om_0)=(B_R,B_1)$ be a trivial solution and let $\cF$, $\cG$ denote the function spaces defined in \eqref{FG}. We know that there exists a finite set of critical values $\widetilde \Si$ such that the following two theorems hold true (compare the following with Theorem \ref{thm II} and \ref{thm III} respectively).
\begin{theorem}[\cite{CY1}]\label{thm IIs}
Let $\sg_c\notin \widetilde\Si$. Then, there exists a threshold $\ve>0$ such that, for all $f\in\cF$ satisfying $\norm{f}_{\cC^{2,\al}}<\ve$ there exists a function $g=g(f,\sg_c)\in\cG$ such that the pair $(D_f,\Om_g)$ is a solution to problem \eqref{eq s}--\eqref{oc s} for $c={|\Om_g|}/{|\pa\Om_g|}$. Moreover, this solution is unique in a small enough neighborhood of $(0,0)\in\cF\times\cG$. In particular, there exist infinitely many nontrivial solutions of problem \eqref{eq}--\eqref{oc}.
\end{theorem}

\begin{theorem}[\cite{CYisaac}]\label{thm IIIs}
Take an element $s_k\in\widetilde\Si$ and set $D=D_0$.
Then $(g,\sg_c)=(0,s_k)\in\cG\times\RR$ is a bifurcation point for the overdetermined problem \eqref{eq s}--\eqref{oc s} in the following sense. There exists a function $\varepsilon\mapsto \la(\varepsilon)\in\RR$ with $\la(0)=0$ 
such that overdetermined problem \eqref{eq s}--\eqref{oc s} admits a nontrivial solution of the form $(D_0,\Om_{g(\varepsilon)})$ for $\sg_c=s_k+\la(\varepsilon)$ and $\varepsilon$ small. 
%If $N=2$, then the symmetry breaking solution $(B_R,\Om_{g(\varepsilon)})$ satisfies 
%\begin{equation}\label{symmetry breaking solutions}
 %   g(\varepsilon) = \varepsilon \cos(m\theta) + o(\varepsilon) \quad \text{as }\varepsilon\to0.   
%\end{equation}
Moreover, there exists a spherical harmonic $Y_k$ of $k$-th degree, such that the symmetry breaking solution $(D_0,\Om_{g(\varepsilon)})$ satisfies 
\begin{equation}\label{symmetry breaking solutions 2}
    g(\varepsilon) = \varepsilon Y_k(\theta) + o(\varepsilon) \quad \text{as }\varepsilon\to0.   
\end{equation}
In particular, there exist uncountably infinitely many nontrivial solutions of problem \eqref{eq s}--\eqref{oc s} where $D$ is a ball (spontaneous symmetry breaking solutions).
\end{theorem}

Finally, we show that, despite showing a very similar local behavior, the overdetermined problems \eqref{eq}--\eqref{oc} and \eqref{eq s}--\eqref{oc s} always give rise to different families of nontrivial solutions. Indeed, as the following result shows, the two overdetermined problems above are ``independent" (in the sense that the only solutions that the two overdetermined problems share are the trivial ones). 
\begin{proposition}
Assume that the pair $(D,\Om)$ satisfies the hypotheses stated in the introduction and solves both overdetermined problems \eqref{eq}--\eqref{oc} and \eqref{eq s}--\eqref{oc s} simultaneously for some constants $d$ and $c$. Then $\Om$ is a ball and $D$ is either empty or a ball concentric with $\Om$. Moreover, $d=\frac{N-1}{N}$ and $c=-\frac{|\pa\Om|}{|\Om|}$.
\end{proposition}
\begin{proof}
By construction, the unique solution $u$ of the boundary value problem \eqref{eq}--\eqref{bc} solves both overdetermined conditions \eqref{oc} and \eqref{oc s}. In particular, $H$ is constant on $\pa\Om$. Aleksandrov's Soap Bubble Theorem (see \cite{Ale1958}) implies that $\Om$ is a ball. Now, if $D$ is not empty, then, by Theorem~\ref{thm saka} (page \pageref{thm saka}) we get that $D$ must be a ball concentric with $\Om$. The rest follows from the explicit expression of $u$ given in \eqref{u}. 
\end{proof}

\section{Appendix}
In what follows, we will construct a subgroup $G^\ast$ of the orthogonal group $O(N)$ that satisfies property \eqref{Y*}.

\begin{definition}
The group $G^*={\rm Id}\times O(N-1)$ is defined as follows.
For all $\ga=(1,\ga')\in G^*$ and $x=(x_1,x')\in\rn=\RR\times\RR^{N-1}$, the element $\ga$ acts on $x$ as 
\begin{equation*}
    \ga(x)=(x_1,\ga'(x')).
\end{equation*}
\end{definition}

\begin{lemma}
Let $P:\rn\to\RR$ be a $G^*$-invariant $k$-homogeneous polynomial. Then the following expressions hold true:
\begin{equation}\label{k odd}
 \text{for $k=2j+1$,}\quad   P(x)=P(x_1,x')=\displaystyle \sum_{i=0}^j a_i\ x_1^{2i+1}\ |x'|^{2(j-i)},
\end{equation}
\begin{equation}\label{k even}
\text{for $k=2j$,}\quad    P(x)=P(x_1,x')=\sum_{i=0}^j a_i\ x_1^{2i}\ |x'|^{2(j-i)},
\end{equation}
for some coefficients $a_0,\dots,a_j\in\RR$.
\end{lemma}
\begin{proof}
Let $P:\rn\to\RR$ be a $G^*$-invariant $k$-homogeneous polynomial. Its terms can by rearranged by factorizing the various powers of $x_1$ whenever they appear. This yields
\begin{equation*}
    P(x)=\sum_{i=0}^k x_1^i P_{k-i}(x'),
\end{equation*}
where the functions $P_{k-i}:\RR^{N-1}\to\RR$ are (possibly zero) $O(N-1)$-invariant $(k-i)$-homogeneous polynomials.
Now, since the polynomials $P_{k-i}$ are $O(N-1)$-invariant, then, in particular, each of them either vanishes or has even degree. Moreover, again by $O(N-1)$-invariance, the restriction of each $P_{k-i}$ to the unit sphere $\SS^{N-2}\subset \RR^{N-1}$ must be a constant, say $c_{k-i}$. By homogeneity, we conclude that 
\begin{equation*}
    P_{k-i}(x')= c_{k-i}|x'|^{k-i}\ \text{ if $k-i$ is even,}\qquad \text{and}\qquad P_{k-i}(x')=0\ \text{ otherwise}.
    %\quad \text{for all }x'\in\RR^{N-1}.
\end{equation*}
Expressions \eqref{k odd} and \eqref{k even} follow.
\end{proof}

\begin{lemma}\label{lemma lapl}
Let $a,b$ be nonnegative integers. Then, the following holds true.
\begin{equation*}
\displaystyle \De \left( x_1^a \ |x'|^{2b}\right)= a(a-1)x_1^{a-2} |x'|^{2b}+ 
    2b(2b-2)x_1^a |x'|^{2b-2}.
\end{equation*}
\end{lemma}
\begin{proof}
We will first compute the gradient of $x_1^a |x'|^{2b}$:
\begin{equation}\label{a gradient}
    \gr \left(x_1^a |x'|^{2b} \right)= a x_1^{a-1}|x'|^{2b} e_1 + 2b x_1^a |x'|^{2b-2}x'.
\end{equation}
Here $e_1$ denotes the vector $(1,0,\dots,0)\in\rn$ and, by a slight abuse of notation, $x'$ also denotes the vector in $\rn$ given by $(0,x')$.

Computing the divergence of \eqref{a gradient} yields
\begin{equation*}
\displaystyle \De \left( x_1^a \ |x'|^{2b}\right)= a(a-1)x_1^{a-2} |x'|^{2b}+ 4ab x_1^{a-1}|x'|^{2b-2}x'\cdot e_1+
    2b(2b-2)x_1^a |x'|^{2b-2}.
\end{equation*}
The claim follows by observing that $e_1\cdot x'=0$. 
\end{proof}

The following proposition implies \eqref{Y*}.
\begin{proposition}
Let $k$ be a natural number. Then, the vector space of $G^*$-invariant $k$-homogeneous harmonic polynomials in $\rn$ %that are also harmonic functions 
is 1-dimensional.
\end{proposition}
\begin{proof}
For simplicity we will only treat the case where $k=2j+1$ is odd, as the case $k=2j$ is analogous.
Let $P:\rn\to\RR$ be a $G^*$-invariant $k$-homogeneous polynomial. By \eqref{k odd}, $P$ can be written as 
\begin{equation*}
    P(x)=\sum_{i=0}^j a_i x_1^{2i+1}|x'|^{2(j-i)}.
\end{equation*}
In other words, we need $j+1$ real coefficients (namely $a_0,\dots,a_j$) to uniquely identify $P$. In what follows, we will show that, under the additional hypothesis that $\De P=0$, only one real coefficient is needed to uniquely identify $P$, that is, the space of $G^*$-invariant $k$-homogeneous polynomials in $\rn$ that are also harmonic functions is 1-dimensional.
Computing the Laplacian of $P$ with Lemma \ref{lemma lapl} at hand yields
\begin{equation*}
    \De P(x)= 
    \sum_{i=0}^j 2i(2i+1)a_i \ x_1^{1+2(j-1)}|x'|^{2(j-i)}+
    \sum_{i=0}^j 2(j-i)(2j-2i-2) a_i\ x_1^{1+2i}|x'|^{2(j-i-1)}.
\end{equation*}
Now, by setting $\ell=i-1$ in the first summation and $\ell=i$ in the second one, we obtain:
\begin{equation*}
    \De P(x)=
    \displaystyle\sum_{\ell=0}^{j-1} \Big( (2\ell+3)(2\ell+2) a_{\ell+1}+2(j-\ell)(2j-2\ell-2)a_\ell \Big) 
    \ x_1^{2\ell+1}|x'|^{2(j-\ell-1)}.
\end{equation*}
Now, since $\De P=0$ by hypothesis, we get the following recursive relations:
\begin{equation*}
    a_{\ell+1}=\frac{-2(j-\ell)(j-\ell-1)}{(2\ell+3)(\ell+1)} \ a_\ell \quad \text{for }\ell=0,\dots,j-1.
\end{equation*}
In other words, all coefficients $a_1,\dots,a_j$ are uniquely determined by the choice of $a_0$. This concludes the proof. The case $k=2j$ is analogous and therefore left to the reader.
\end{proof}

%Overdetermined problem \eqref{eq}--\eqref{oc}
%\begin{equation*}
%    d>\frac{N-1}{N} \quad d=\frac{N-1}{N} \quad d<\frac{N-1}{N}
%\end{equation*}
%No solutions  
%
%Only trivial solution 
%
%Nontrivial solutions
%
%Symmetry breaking solutions
%
%$\Om=$ball is an isolate solution 
%
%\begin{equation*}
%    D\ne\emptyset\quad D=0\quad \sg_c\notin\Si\quad \sg_c\in\Si
%\end{equation*}
%
%Theorem \ref{thm I}
%
%Theorem \ref{thm II}
%
%Theorem \ref{thm III}
%
%Theorem \ref{thm IV}

\section*{Acknowledgements}
The author would like to express his gratitude to Giorgio Poggesi (The University of Western Australia) for bringing this problem to his attention.

\begin{small}

\end{small}

\bigskip

\noindent
\textsc{
Mathematical Institute, Tohoku University, Aoba, 
Sendai 980-8578, Japan } \\
\noindent
{\em Electronic mail address:}
cavallina.lorenzo.e6@tohoku.ac.jp

\end{document}